\documentclass{amsart}
\usepackage{amsmath,amssymb,amsxtra,amsthm}
\usepackage{graphicx}

\theoremstyle{plain}
\newtheorem{theorem}{\bf Theorem}[section]
\newtheorem*{theorem*}{\bf Theorem}
\newtheorem*{lemma*}{\bf Lemma}

\newtheorem*{proposition*}{\bf Proposition}
\theoremstyle{remark}

\newtheorem{conjecture}{\bf Conjecture}
\newtheorem*{conjecture*}{\bf Conjecture}
\theoremstyle{remark}

\numberwithin{equation}{section}

\def\N{{\mathbb N}}

\def\Z{{\mathbb Z}}

\newcommand{\BO}{\mathnormal O}

\title{ Sign Changes of the Liouville function on Quadratics}
\author{Peter Borwein, Stephen K.K. Choi and Himadri Ganguli}
\address{Department of Mathematics, Simon Fraser University, Burnaby, British Columbia, CANADA V5C 1S6.}
\email{pborwein@sfu.ca}
\email{kkchoi@math.sfu.ca}
\email{hganguli@sfu.ca}
\thanks{Research of first and second  authors are supported by NSERC of Canada.}
\date{\today}
\begin{document}
\maketitle
\begin{abstract}
Let $\lambda (n)$ denote the Liouville function. Complementary to the prime number theorem, Chowla conjectured
that
\vspace{1mm}

\noindent {\bf Conjecture (Chowla).} {\em
\begin{equation}
\label{a.1}
\sum_{n\le x} \lambda (f(n)) =o(x)
\end{equation}
for any polynomial $f(x)$ with integer coefficients which is not of form $bg(x)^2$. }
\vspace{1mm}

\noindent The prime number theorem is
equivalent to \eqref{a.1} when $f(x)=x$. Chowla's conjecture is proved for linear functions but for the degree
greater than 1, the conjecture seems to be extremely hard and still remains wide open. One can consider a weaker form
of Chowla's conjecture, namely,
\vspace{1mm}

\noindent {\bf Conjecture 1 (Cassaigne, et al).} {\em If $f(x) \in \Z [x]$ and is not in the form of $bg^2(x)$
for some $g(x)\in \Z[x]$, then $\lambda (f(n))$ changes sign infinitely often.}

\vspace{1mm}

\noindent Clearly, Chowla's conjecture implies Conjecture 1. Although it is weaker, Conjecture 1 is still wide open for polynomials of degree $>1$. In this article, we study Conjecture 1 for the
quadratic polynomials. One of our main theorems is

\vspace{1mm}

\noindent {\bf Theorem 1.} {\em Let $f(x) = ax^2+bx +c $ with $a>0$ and $l$ be a positive integer such that $al$ is not a perfect square. Then if the equation $f(n)=lm^2 $ has one solution $(n_0,m_0) \in \Z^2$, then it has infinitely many positive solutions $(n,m) \in \N^2$.}

\vspace{1mm}

\noindent As a direct consequence of Theorem 1, we prove

\vspace{1mm}

\noindent {\bf Theorem 2.} {\em Let $f(x)=ax^2+bx+c$ with $a \in \N$ and $b,c \in \Z$. Let
\[
A_0=\left[ \frac{|b|+(|D|+1)/2}{2a}\right]+1.
\]
Then the binary sequence $\{ \lambda (f(n)) \}_{n=A_0}^\infty$ is either a constant sequence or it changes sign infinitely often.}

\vspace{1mm}

\noindent Some partial results of Conjecture 1 for quadratic polynomials are also proved by using Theorem 1.
\end{abstract}

\section{Introduction}

Let $\lambda(n)$ denote the Liouville function, i.e, $\lambda(n) = (-1)^{\Omega(n)}$, where
$\Omega(n)$ denotes the number of prime factors of $n$ counted with multiplicity. Alternatively, $\lambda (n)$ is
the completely multiplicative function defined by $\lambda (p)=-1$ for each prime. Let $\zeta (s)$ denote the Riemann zeta function, defined for complex $s$ with $\Re (s) > 1$ by
\[
\zeta (s):=\sum_{n=1}^\infty \frac{1}{n^s} =\prod_{p}\left( 1-\frac{1}{p^s} \right)^{-1}
\]
where the product is over all prime numbers $p$.
Thus
\begin{equation}
\label{a.1.1}
\frac{\zeta (2s)}{\zeta (s)} = \prod_{p}\left( 1+\frac{1}{p^s} \right)^{-1}
   =\prod_{p}\left( 1-\frac{\lambda (p) }{p^s} \right)^{-1}  = \sum_{n=1}^\infty \frac{\lambda (n)}{n^s}.
\end{equation}
   Let $L(x)$ denote the average of the values of $\lambda (n)$ up to $x$,
\[
L(x):=\sum_{1\le n \le x}\lambda (n)
\]
so that $L(x)$ records the difference of the number of positive integers up to $x$ with an even number of prime
factors (counted with multiplicity) and those with an odd number. P\'{o}lya in 1919 showed in \cite{Po} that
the Riemann Hypothesis, i.e., every non-trivial zeros of $\zeta (s)$ are on the critical line $\Re (s)=1/2$,
will follow if $L(x)$ does not change sign for sufficiently large $n$. There is a vast amount  of literature
about the study of the sign change of $L(x)$. Until 1958, Haselgrove proved that $L(x)$ changes sign infinitely often
in \cite{HA}. For more discussion about this problem, we refer the reader to \cite{BFM}.

It is well known that the prime number theorem is equivalent to
\begin{equation}
\label{1.1}
L(x)=\sum_{n\le x} \lambda (n) =o(x).
\end{equation}
In fact, the prime number theorem is equivalent to that fact that $\zeta (s) \neq 0$ on the vertical line $\Re (s)=1$; and this is equivalent to \eqref{1.1} in view of \eqref{a.1.1}.
Complementary to the prime number theorem, Chowla \cite{Chowla} made the following conjecture
\begin{conjecture}[Chowla]
\label{chowlaconj}
Let $f(x)\in \Z [x]$ be any polynomial which is not of form $bg^2(x)$ for some $b\neq 0, g(x) \in \Z[x]$. Then
\begin{equation}
\label{1.2}
\sum_{n \leq x} \lambda(f(n))=o(x).
\end{equation}
\end{conjecture}
Clearly, Chowla's conjecture is equivalent to the prime number theorem when $f(x)=x$. For polynomials of degree
$>1$, Chowla's conjecture seems to be extremely hard and still remains wide open. One can consider a weaker form
of Chowla's conjecture, namely,
\begin{conjecture}[Cassaigne et al.]
\label{casconj}
If $f(x) \in  \Z[x]$ is not of form $bg^2(x)$, then $\lambda (f(n))$ changes sign infinitely often.
\end{conjecture}
Clearly, Chowla's conjecture implies Conjecture \ref{casconj}.  In fact, suppose it is not true, i.e., there is $n_0$ such that $\lambda (f(n)) =\epsilon$ for all $n \ge n_0$ where $\epsilon$ is either $-1$ or $+1$. Then it follows that
\[
\sum_{n\le x} \lambda (f(n)) = \epsilon x + \BO (1)
\]
which contradicts \eqref{1.2}.

Although it is weaker, Conjecture \ref{casconj} is still wide open for polynomials of degree $>1$. In \cite{Cas}, Conjecture  \ref{casconj} for special polynomials have been studied and some partial results are proved.

\begin{theorem}[Cassaigne et al.]
\label{cas1}
Let $f(n) = (an+b_{1})(an+b_{2})\dots(an+b_{k})$  where $a,k \in \N $, $b_{1}, \ldots,b_{k}$ are distinct integers with $b_{1}\equiv \ldots \equiv b_{k} \pmod a$  then  $\lambda(f(n))$ changes sign infinitely often.
\end{theorem}
\begin{proof}
This is Corollary 2 in \cite{Cas}.
\end{proof}

For certain quadratic polynomials, they proved
\begin{theorem}[Cassaigne et al.]
\label{cas2}
If $f(n) = (n+a)(bn+c)$ where  $a ,c \in \Z ,\ b \in \N , ab \neq c $ then  $\lambda(f(n))$ changes sign infinitely often.
\end{theorem}
\begin{proof}
This is Theorem 4 in \cite{Cas}.
\end{proof}
\begin{theorem}[Cassaigne et al.]
\label{cas3}
Let $a \in \N, b,c \in \Z$, and write $f(n)=an^2+bn+c$, $D=b^2-4ac.$ Assume that $a,b$ and $c$ satisfy the following conditions :
\begin{enumerate}
\item[(i)] $2a | b$,
\item[(ii)] $D<0$,
\item[(iii)] there is a positive integer $k$ with
$$ \lambda\left( -\frac{D}{4}k^2 +1\right) = -1.$$
\end{enumerate}
Then $\lambda(f(n))$ changes sign infinitely often.
\end{theorem}
\begin{proof}
This is Theorem 3 in \cite{Cas}.
\end{proof}

In this article,  we continue to study Conjecture \ref{casconj} for the quadratic case. One of our main results is Theorem \ref{thm 2.2} below. By Theorem \ref{thm 2.2}, in order to show that the sequence $\{ \lambda (f(n)) \}_{n=1}^\infty$ changes sign infinitely often, we only need find one pair of large integers $n_1$ and $n_2$ such
that $\lambda (f(n_1)) \neq \lambda (f(n_2))$. This will make the conjecture much easier to handle. Some partial results from Theorem \ref{thm 2.2} are also deduced in the next section.

\section {Main Results}

Conjecture \ref{casconj} for the linear polynomial is easily solved by the following result.
\begin{theorem}
Let $P := \{ n \in \N : \lambda(n)= +1 \} $ and $ N:= \{ n \in \N : \lambda(n) = -1 \}. $ Then both $P$ and $N$ cannot contain infinite arithmetic progression.
In particular, $\lambda (an+b)$ changes sign infinitely often in $n$.
\end{theorem}
\begin{proof}
We claim that both P and N cannot contain any infinite arithmetic progression.
Suppose not and there are an $n_{0}$ and $l$ such that
\begin{equation}
\lambda(n_{0}+lk)=\lambda(n_{0})
\end{equation}
for $k=0,1,2,\dots$. Pick a prime $p$  which is of the form $lm+1$. Now put $k=mn_{0}$ and consider
$$ \lambda(n_{0}+lk)=\lambda(n_{0}+ lmn_{0})=\lambda(n_{0})\lambda(lm+1)=\lambda(n_{0})\lambda(p)=-\lambda(n_{0}).$$
This contradicts (1). Hence our claim is attained.
\end{proof}

One of the main results in this paper is the following theorem.
\begin{theorem}
\label{thm 2.2}
Let $f(x) = ax^2+bx +c $ with $a>0$ and $l$ be a positive integer such that $al$ is not a perfect square. Then if the equation $f(n)=lm^2 $ has one solution $(n_0,m_0) \in \Z^2$, then it has infinitely many positive solution $(n,m) \in \N^2$.
\end{theorem}
\begin{proof}
Let $D=b^2-4ac$ be the discriminant of $f(x)$. By solving the quadratic equation
\begin{equation}
\label{2}
an^2+bn+c=lm^2,
\end{equation}
for $n$ we get
$$ n_0= {-b \pm \sqrt{b^2-4a(c-lm_0^2)} \over 2a } = {-b \pm \sqrt{D+4alm_0^2} \over 2a }.$$
It follows that $D+4alm_0^2=t_0^2$ for some integer $t_0$. By choosing a suitable sign of $t_0$, we may assume
\begin{equation}
\label{2.0}
t_0\equiv b \pmod {2a}, \ \ \ \  \mbox{ and } \ \ \ \ n_0=\frac{-b+t_0}{2a}\in \Z.
\end{equation}
This leads us to consider the diophantine equation
\begin{equation}
\label{2.1}
t^2=4alm^2+D.
\end{equation}
Suppose that $(t_0,m_0)$ and $(t,m)$ are solutions of \eqref{2.1}. Then we have
\begin{equation*}
t^2=4alm^2+D
\end{equation*}
and
\begin{equation*}
t_{0}^2=4alm_{0}^2+D.
\end{equation*}
Subtracting the above two equations, we get
\begin{equation}
(t-t_{0})(t+t_{0})=l(m-m_{0})(4am+4am_{0}).
\end{equation}
We now let $s$ and $r$ be
\begin{equation}
\label{2.2}
r(m-m_{0})=2as(t+t_{0})  \  , \   2as(4alm+4alm_{0})=r(t-t_{0}).
\end{equation}
By eliminating the terms $t$ and $m$ respectively in \eqref{2.2}, we get
\begin{equation}
\label{2.3}
(r^2-16a^3ls^2)m=r^2m_{0}+16a^3ls^2m_{0}+4arst_{0}
\end{equation}
and
\begin{equation}
\label{2.4}
(r^2-16a^3ls^2)t=r^2t_{0} +16a^2ls^2m_{0}+16a^3s^2lt_{0}.
\end{equation}
Note that by our assumption, $16a^3l$ is not a perfect square.
So the Pell equation
\begin{equation}
\label{2.5}
r^2-16a^3ls^2=1
\end{equation}
always have infinitely many solutions $(r,s) \in \Z^2$, and for each solution $(r,s)$ of the Pell equation \eqref{2.5} gives integers $m$ and $t$ through \eqref{2.3} and \eqref{2.4} such that
\begin{equation*}
m=r^2m_{0}+16a^3ls^2m_{0}+4arst_{0}
\end{equation*}
and
\begin{equation*}
t=r^2t_{0}+16a^2lsrm_{0}+16a^3s^2lt_{0}.
\end{equation*}
One can easily verify that if $(r,s)\neq (\pm 1,0)$ then these $m$ and $t$ satisfy the equations \eqref{2.2} and hence satisfy equation
\eqref{2.1}. Note that $r^2 \equiv 1 \pmod{2a}$ and $r(m-m_{0}) \equiv 0 \pmod{2a}$. Hence we have $m \equiv
m_{0} \pmod{2a}$ and $t \equiv t_{0} \pmod{2a}$ by \eqref{2.2}. Since there are infinitely many solutions $(r,s) \in \Z^2$ of
the Pell equation \eqref{2.5} and these will give infinitely many solutions $(m,t) \in \Z^2$
of the equation \eqref{2.2}. In particular, there are infinite many positive  integers $t$ such that $t\equiv t_0\pmod {2a}$ and
\[
n= {-b + \sqrt{D+4alm^2} \over 2a } = {-b + t \over 2a }
\]
to be a positive integer by \eqref{2.0}. Therefore, there are infinitely many positive solutions $(n,m) \in \N^2$ of \eqref{2}.
This completes the proof of the theorem.
\end{proof}

It is worth to mention that one should not expect Theorem \ref{thm 2.2} is true for polynomials of higher degree because there may only have finitely many integer solutions by Siegel's theorem on integral points in \cite{Si}.

In view of Theorem \ref{thm 2.2}, to determine the conjecture is true for a given quadratic polynomial $f(x)$, we only need to find one pair of positive integers $n_1$ and $n_2$ such that $\lambda (f(n_1)) \neq \lambda (f(n_2))$. This gives us the following theorem.

\begin{theorem}
\label{thm 2.3}
Let $f(x)=ax^2+bx+c$ with $a \in \N$ and $b,c \in \Z$. Let
\[
A_0=\left[ \frac{|b|+(|D|+1)/2}{2a}\right]+1.
\]
Then the binary sequence $\{ \lambda (f(n)) \}_{n=A_0}^\infty$ is either a constant sequence or it changes sign infinitely often.
\end{theorem}
\begin{proof}
Suppose $\{ \lambda (f(n)) \}_{n=A_0}^\infty$ is not a constant sequence. Then there are positive integers
$n_1\neq n_2 \ge A_0$ such that $\lambda (f(n_1)) \neq \lambda (f(n_2))$. Hence there are positive integers
$l_1,l_2$ and $m_1,m_2$ such that
\begin{equation}
\label{3.3}
\lambda (l_1)=+1, \ \ \ \ \mbox{ and } \ \ \ \ \lambda (l_2)=-1,
\end{equation}
and
\[
f(n_1)=l_1m_1^2, \ \ \ \ \mbox{ and } \ \ \ \ f(n_2)=l_2m_2^2.
\]
We claim that $al_1$ and $al_2$ are not perfect squares. If $al_j=N^2$ is a perfect square, then the diophantine equation $t^2=D+4al_jm^2$ has only finitely many solutions $(t,m)$. In fact,  since $(t_j-2Nm_j)(t_j+2Nm_j)=D$, so
there is $d\neq 0$ such that $t_j+2Nm_j=D/d$ and $t_j-2Nm_j=d$. It follows that $2t_j=D/d +d$. Thus,
\[
|t_j| \le\frac12 \left( \frac{|D|}{|d|} +|d| \right) \le \frac{|D|+1}{2}.
\]
Since $f(n_j)=l_jm_j^2$, so
\[
n_j = \left| \frac{-b \pm \sqrt{D+4al_jm_j}}{2a}\right|
\le  \frac{|b|+|t_j|}{2a}\le \frac{|b|+(|D|+1)/2}{2a}< A_0.
\]
This contradicts $n_j \ge A_0$. Therefore from Theorem \ref{thm 2.2}, there are infinitely many $n_1$ and $n_2$
such that $\lambda (f(n_1)) \neq \lambda (f(n_2))$ and hence $\lambda (f(n))$ changes sign infinitely often.
\end{proof}
As we remarked above, one should not expect Theorem \ref{thm 2.3} to be true for polynomials of higher degree.

We prove some partial results of special quadratic polynomials.
\begin{theorem}
Let $f(n)=n^2+bn+c$ with $b,c \in \Z$ . Suppose  there exists a positive integer $n_{0} \ge A_0$ (with $a=1$) such that $\lambda(f(n_{0})) = -1 $. Then $\lambda(f(n))$ changes sign infinitely often.
\end{theorem}
\begin{proof}
We observe the following identity
\[
f(n)f(n+1)=f(f(n)+n)
\]
which can be verified directly. Hence we have
\begin{equation}
\label{3.1}
\lambda (f(n)) \lambda (f(n+1)) = \lambda (f(f(n)+n)).
\end{equation}
If $\{ \lambda (f(n)) \}_{n=1}^\infty$ is a constant sequence, then it follows from \eqref{3.1} that
\[
\lambda (f(n)) =+1,\ \ \ \  \mbox{ for all $n=1,2,\ldots $.}
\]
Therefore if there is $n_0 \ge A_0$ such that $\lambda (f(n_0))=-1$, then by Theorem \ref{thm 2.3}, $\{ \lambda (f(n)) \}_{n=1}^\infty$
  changes sign infinitely often. This proves the theorem.
\end{proof}

The proof of Theorem \ref{thm 2.2} shows that the solvability of  the diophantine equation
\begin{equation}
\label{4.0}
X^2 - 4al Y^2 =D
\end{equation}
is critical in solving the problem. In general, there is no simple criterion to determine the solvability of the equation \eqref{4.0} except if we know the central norm of the continued fraction of the irrational number $\sqrt{al}$. For more discussion in this area, we refer the readers to \cite{Mo1}-\cite{Mo4}. The following theorem deals with a special case of $D$ for which we can solve the equation \eqref{4.0}.
\begin{theorem}
Let $f(x)=px^2+bx+c$ with prime number $p$. Suppose the discriminant $D=b^2-4pc$ is a non-zero perfect square.  Then $\lambda(f(n))$ changes sign infinitely often.
\end{theorem}
\begin{proof}
We first choose positive integers $l_1$ and $l_2$ such that $p\, l_1$ and $p\, l_2$ are not perfect squares and
$\lambda (l_1) \neq \lambda (l_2)$. So the Pell equations
\begin{equation}
\label{3.2}
X^2 -4p\, l_j Y^2=1, \ \ \ \ j=1,2
\end{equation}
have infinitely many positive solutions $(X,Y)$. Let $D=q^2$ with $q \ge 1$. Then any positive solution $(X,Y)$ of \eqref{3.2} gives a positive solution $(qX,qY)$ of
\[
X^2 -4p\, l_jY^2=D.
\]
We can choose $X$ large enough so that $-b +qX >0$. On the other hand, we have $X^2\equiv 1 \pmod p$ by
\eqref{3.2} and $q^2 \equiv b^2 \pmod p$ because $D=b^2-4pc$. Therefore $(qX)^2\equiv b^2 \pmod p$. Since $p$
is a prime, so either (a) $qX \equiv b \pmod p$ or (b)
$qX \equiv -b \pmod p$.
We define
\[
n=\frac{-b \pm qX}{2p}
\]
where the sign $\pm$ is determined according to Cases (a) or (b) so that $n$ is a positive integer. Therefore $(n,qX)$ is a positive solution of the equations $f(n)=l_jm^2$. Then our theorem follows readily from Theorem \ref{thm 2.2}.
\end{proof}

\end{document}